\documentclass[12pt]{amsart}
\usepackage{amssymb,amsmath,amsthm,latexsym}
\usepackage{mathrsfs}
\usepackage{a4wide}
\usepackage[usenames,dvipsnames]{color}
\usepackage{euscript}
\usepackage{graphicx}
\usepackage{mdwlist}
\usepackage{mathtools,dsfont,wasysym}
\usepackage[dvipsnames]{xcolor}

\DeclareFontFamily{U}{mathx}{\hyphenchar\font45}
\DeclareFontShape{U}{mathx}{m}{n}{
      <5> <6> <7> <8> <9> <10>
      <10.95> <12> <14.4> <17.28> <20.74> <24.88>
      mathx10
      }{}

\newcommand{\nn}[1]{{\vert\kern-0.25ex\vert\kern-0.25ex\vert #1 
    \vert\kern-0.25ex\vert\kern-0.25ex\vert}}

\newtheorem{theorem}{Theorem}[section]
\newtheorem*{theorema}{Theorem A}
\newtheorem*{theoremb}{Theorem B}

\newtheorem{lemma}[theorem]{Lemma}
\newtheorem{corollary}[theorem]{Corollary}
\newtheorem{proposition}[theorem]{Proposition}
\theoremstyle{remark}
\newtheorem{remark}[theorem]{Remark}
\theoremstyle{definition}
\newtheorem{definition}[theorem]{Definition}

\numberwithin{equation}{section}
\makeatother

\newcommand{\vertiii}[1]{{\left\vert\kern-0.25ex\left\vert\kern-0.25ex\left\vert #1 
    \right\vert\kern-0.25ex\right\vert\kern-0.25ex\right\vert}}

\newcounter{smallromans}

\newenvironment{romanenumerate}
{\begin{list}{{\normalfont\textrm{(\roman{smallromans})}}}%
  {\usecounter{smallromans}\setlength{\itemindent}{0cm}%
   \setlength{\leftmargin}{5.5ex}\setlength{\labelwidth}{5.5ex}%
   \setlength{\topsep}{.5ex}\setlength{\partopsep}{.5ex}%
   \setlength{\itemsep}{0.1ex}}}%
{\end{list}}

\newcounter{smallromansdash}

{\end{list}}

\newcounter{bigromans} 
  {\end{list}}

\begin{document}


\baselineskip=17pt



\title[Ornstein--Uhlenbeck processes driven by a L\'evy process]{Law equivalence of Ornstein--Uhlenbeck processes driven by a L\'evy process}
\author[G.~Bartosz]{Grzegorz Bartosz}
\address{Korczaka 6, 43-100 Tychy, Poland}
\email{grzegorzjanbartosz@wp.pl}
\author[T.~Kania]{Tomasz Kania}
\address{Institute of Mathematics, Czech Academy of Sciences, \v{Z}itn\'{a} 25, 115~67 Prague 1, Czech Republic}
\email{tomasz.marcin.kania@gmail.com}
\thanks{The second-named author acknowledges with thanks funding received from GA\v{C}R project 19-07129Y; RVO 67985840 (Czech Republic).}

\date{\today}

\begin{abstract}We demonstrate that two Ornstein--Uhlenbeck processes, that is, solutions to certain stochastic differential equations that are driven by a L\'evy process $L$ have equivalent laws as long as the eigenvalues of the covariance operator associated to the Wiener part of $L$ are strictly positive. Moreover, we show that in the case where the underlying L\'evy process is a purely jump process, which means that  neither it has a Wiener part nor the drift, the absolute continuity of the law of one solution with respect to another forces equality of the solutions almost surely.\end{abstract}


\maketitle
\section{Introduction and the main result}
Let $H$ be a real separable Hilbert space and let $L=(L_t)_{t\geqslant 0}$ be an $H$-valued L\'evy process. We fix two bounded linear operators $A$, $\tilde{A}$ on $H$; should the space $H$ have finite dimension $d$, we think of $A$ and $\tilde{A}$ as $d\times d$ real matrices in some fixed basis. We consider the corresponding \emph{$H$-valued Ornstein--Uhlenbeck processes driven by a L\'evy process}, that is, solutions to the following stochastic differential equations:
\begin{equation}\label{equations}\left\{\begin{array}{lcl} {\rm d}X_t &=& A X_t\, {\rm d}t + {\rm d}L_t,\\ {\rm d}\tilde{X}_t &=& \tilde{A} \tilde{X}_t\, {\rm d}t + {\rm d}L_t\end{array}\right. \end{equation}
with the initial conditions $X_0 = \tilde{X}_0 = 0$. (The reader will find all the required definitions as well as proofs of the results to be presented in the subsequent sections.)\smallskip

It is customary to view sample paths of such processes (truncated to some initial interval $[0,T]$ for $T>0$) as elements of the Hilbert space $H_T:=L_2\big([0,T], H\big)$. On the other hand, we may be more restrictive and regard these processes as random variables assuming values in $\mathcal{D}_{H,T}$, the space of $H$-valued c\`adl\`ag functions on $[0,T]$ furnished with the Skorohod topology. This topology is induced by the so-called \emph{Skorohod metric}; the Borel $\sigma$-algebra of $\mathcal{D}_{H,T}$ coincides then with the cylindrical $\sigma$-algebra, that is, the smallest $\sigma$-algebra making the point evaluations $$p_t(f)=f(t)\quad (t\in [0,T], f\in \mathcal{D}_{H,T})$$ measurable--this is an important feature unavailable in the Hilbert space $H_T$. A natural question then arises. \begin{quote}\emph{In which circumstances such two processes have equivalent laws?}\end{quote} This line of research concerning the study of equivalence of laws  was initiated by Kozlov (\cite{kozlov}) in the setting where $A,\tilde{A}$ are elliptic and self-adjoint operators on a smooth manifold without boundary and the equations are driven by a Brownian motion. This theory was developed further by Zabczyk (\cite{Zabczyk}) in much greater generality (see also the seminal monograph \cite{DZ}), Peszat (\cite{PeszatGoldys, Peszat2}), and other authors (\cite{BvN}, \cite{MvN}). \smallskip

The aim of this paper is to extend and complement already existing results for Ornstein--Uhlenbeck processes driven by a (cylindrical) Wiener processes that take values in a finite- or infinite-dimensional Hilbert space to Ornstein--Uhlenbeck processes driven by L\'evy processes that possibly have jumps. Our results appear to be new also in the case where $H = \mathbb R^d$ for some $d\in \mathbb N$. \smallskip

We fix a real separable Hilbert space $H$ and a~$H$-valued L\'evy process $L=(L_t)_{t\geqslant 0}$ on some probability space $(\Omega, \mathcal {F}, \mathsf P)$ that is expressed in the L\'evy--It\^o decomposition as $L_t = bt + W_t + Z_t$ ($t\geqslant 0$), where $b\in H$, $W=(W_t)_{t\geqslant 0}$ is a (possibly degenerate) Wiener process with the covariance operator $Q$, and $(Z_t)_{t\geqslant 0}$ is the jump part of $L$ (see Theorem~\ref{levyito} for more details).
 
\begin{theorema}Let $T>0$. Suppose that the eigenvalues of the covariance operator $Q$ corresponding to $W$ are strictly positive. Let $X,\tilde{X}\colon \Omega\to \mathcal{D}_{H,T}$ be the Ornstein--Uhlenbeck processes solving \begin{equation}\label{equations4}\left\{\begin{array}{rcl} {\rm d}X_t &=& A X_t\, {\rm d}t + {\rm d}L_t,\\ {\rm d}\tilde{X}_t &=& \tilde{A} \tilde{X}_t\, {\rm d}t + {\rm d}{L}_t, \\ X_0 & = &  0, \\ \tilde{X}_0 &=& 0.\end{array} \right. \end{equation}If $H$ is finite-dimensional or \begin{equation}\label{technical}\int\limits_0^t \, \big(AX_s - \tilde{A}X_s\big){\rm d}s \in {\rm im}\, Q^{1/2}\quad \big(t\in [0,T]\big),  \end{equation} then the laws of $X$ and $\tilde{X}$ are equivalent.\end{theorema}

\noindent \emph{Note}. When the Hilbert space $H$ is finite-dimensional, the main hypothesis of Theorem A is equivalent to invertibility of the covariance matrix $Q$ corresponding to $W$, in which case \eqref{technical} holds vacuously as $Q^{1/2}$ is invertible. Otherwise, $Q^{1/2}$ is only a densely defined positive operator as $Q$ is trace-class.\medskip

In the case where the underlying L\'evy process $L$ is a purely jump process, solutions to \eqref{equations4} with respect to $L$ exhibit a remarkably rigid behaviour.

\begin{theoremb}Let $T>0$. Suppose that $L$ is a purely jump process, that is, $L = Z$ as written in the L\'evy--It\^o decomposition. Let $X_J,\tilde{X}_J\colon \Omega\to \mathcal{D}_{H,T}$ be the Ornstein--Uhlenbeck processes solving \eqref{equations4}. If the law of $X_J$ is absolutely continuous with respect to the law of $\tilde{X}_J$, then the processes $X_J$ and $\tilde{X}_J$ are equal to each other almost surely.\end{theoremb}

\section{Preliminaries} Let $(\Omega, \mathcal F, \mathsf P)$ be a probability space and let $(S, \mathcal S)$ be a measurable space. A function $X\colon \Omega \to S$ is an $S$-\emph{valued random variable} when $X^{-1}(E)\in \mathcal F$ for every $E\in \mathcal S$. The \emph{law} of $X$ is the pull-back measure $\mathsf P_X$ on $(S, \mathcal{S})$ given by $\mathsf P_X(E) = \mathsf P(X^{-1}(E))$ ($E\in \mathcal S$). When $S$ carries the structure of a metric space, by default we will take $\mathcal S= \mathsf{Bor}\, S$, the $\sigma$-algebra of Borel subsets of $S$. For two measures $\mu$ and $\nu$ we denote by $\mu \otimes \nu$ the product measure defined on the product $\sigma$-algebra, that is, the smallest $\sigma$-algebra containing all measurable rectangles from the respective measure spaces. For a separable metric space $S$, or more generally, a second-countable Hausdorff space, the product $\sigma$-algebra $\mathsf{Bor}\, S \otimes \mathsf{Bor}\, S$ coincides with $\mathsf{Bor}\, S\times S$ (see, \emph{e.g.}, \cite[Lemma 6.4.2]{Bogachev}).\smallskip

\subsection{The Skorohod metric} Let $(S, d)$ be a separable metric space and let $T>0$ be given. Let $\mathcal{D}_{S,T}$ denote the space of all $S$-valued c\`adl\`ag functions, that is, right-continuous functions $f\colon [0, T]\to S$ with the property that for each $t > 0$ the left limit at $t$, $f(t-)$, exists. Denote by $\Lambda_T$ the family of all strictly increasing functions $\phi$ from $[0,T]$ onto itself with $\phi(0)=0$ and $\phi(T)=T$. Then the formula
$$d_S(f,g)=\inf_{\phi\in \Lambda_T} \max\big\{\sup_{t\in [0,T]}|\phi(t)-t|, \sup_{t\in [0,T]} d\big(f(\phi(t), g(t) \big) \big\}\quad \big(f,g\in \mathcal{D}_{S,T}\big) $$
defines a metric on $\mathcal{D}_{S,T}$, called the \emph{Skorohod metric} (\cite[p.~265 \& Proposition 1.6]{Jakubowski}).
The Borel $\sigma$-algebra of the space of c\`adl\`ag functions with the Skorohod metric, $\mathsf{Bor}\, \mathcal{D}_{S,T}$, coincides with the $\sigma$-algebra of cylindrical sets--in other words, it is the smallest $\sigma$-algebra on $\mathcal{D}_{S,T}$ for which the point evaluations $p_t(f)=f(t)$ $(t\in [0,T], f\in \mathcal{D}_{S,T})$ are measurable (\cite[Corollary 2.4]{Jakubowski}). We shall frequently invoke the following consequence of this fact.

\begin{proposition}\label{usefulprojections}Let $(\Omega, \mathcal{F})$ be a measurable space. Then a function $X\colon \Omega\to \mathcal{D}_{S,T}$ is measurable if and only if, for every $t\in [0,T]$ the composite map $\pi_t\circ X$ is measurable.  \end{proposition}

\subsection{Absolute continuity of measures} Let $\mu, \nu$ be measures on a measurable space $(S, \mathcal S)$. The measure $\mu$ is called \emph{absolutely continuous} with respect to $\nu$ (in short, $\mu \ll \nu$), when $\nu(E)=0$ ($E\in \mathcal S$) implies that $\mu(E)=0$. Two measures are \emph{equivalent} when they are mutually absolutely continuous. Let us record the following corollary to Fubini's theorem concerning absolute continuity of product measures (\cite[p.~92]{Folland}). Suppose that $\mu_i, \nu_i$ are $\sigma$-finite measures on measurable spaces $(S_i, \mathcal S_i)$ such that $\mu_i \ll \nu_i$ ($i=1,2$). Then $\mu_1\otimes \mu_2 \ll \nu_1 \otimes \nu_2$. We will make use of this fact stated in the following form.
\begin{lemma}\label{Lemma2new}Let $(\Omega, \mathcal F, \mathsf P)$ be a probability space and let $(S, \mathcal S)$ be a measurable space. Suppose that $X,Y_1,Y_2\colon \Omega\to S$ are random variables such that
\begin{romanenumerate}
\item $\mathsf P_{Y_1} \ll \mathsf P_{Y_2}$,
\item the variables $X, Y_i$ are independent $(i=1,2)$.\end{romanenumerate}
Then the law $\mathsf P_{(X, Y_1)}$ is absolutely continuous with respect to $\mathsf P_{(X, Y_2)}$.\end{lemma}
\begin{proof}Independence of $X$ and $Y_i$ is equivalent to $\mathsf P_{(X, Y_i)} = \mathsf P_X \otimes \mathsf P_{Y_i}$ ($i=1,2$) (see, \emph{e.g.}, \cite[Th\'eor\`eme IV.1.3]{metiver}).\end{proof}

\subsection{L\'evy processes} Let $B$ be a separable Banach space. A stochastically continuous, $B$-valued process $L = (L_t)_{t\geqslant 0}$, is \emph{L\'evy}, when $L_0 = 0$ almost surely, the increments of $L$ are independent and stationary, and almost every sample path $f(t):=L_t(\omega)$ ($\omega\in \Omega$) of $L$ is a~$B$-valued c\`adl\`ag function.\smallskip

Let $E\in \mathsf{Bor}\, B\setminus \{0\}$, $t\geqslant 0$ and $f\in \mathcal{D}_B$. We then define
\begin{equation}\label{Poissonmeasure}\pi_t(E, f) = {\rm card}\{s\leqslant t\colon \Delta f(s):=f(s) - f(s-) \in E \}.\end{equation}
For a L\'evy process  $L = (L_t)_{t\geqslant 0}$ we may then set \begin{equation}\label{piL}\pi_t(E, L)(\omega)=\pi_t(E, f)\quad (\omega\in \Omega),\end{equation} where $f(t)=L_t(\omega)$ is a sample path of $L$. Put simply, $\pi_t(E, L)(\omega)$ counts the number of jumps in $E$ up to time $t$ that the sample path of $L$ at $\omega$ has in the set $E$. The family $\{\pi_t(\cdot, L)\colon t\geqslant 0\}$ is called \emph{the Poisson random measure} of $L$. The formula
$$\mu(E) = \mathsf E \,\left[\pi_t(E, L)\right] \quad (E\in \mathsf{Bor}\, B\setminus \{0\})$$
defines a Borel measure on $B\setminus \{0\}$, called the \emph{intensity measure} of $L$. For a \emph{bounded below} Borel set $E\subset B$, that is a set with ${\rm dist}(0,E) > 0$, and $f\in \mathcal{D}_B$ we set $$\hat{\pi}_t(E,f) = \pi_t(E, f) - t \cdot \mu(E)\quad (t\geqslant 0).$$ Whenever $E$ is bounded below, the expression
\begin{equation}\label{z1}Z^1_E(f,t):= \sum_{\begin{smallmatrix}0\leqslant s \leqslant t\\ \Delta f(s)\in E\end{smallmatrix}} \Delta f(s) = \int\limits_E u \,\pi_t({\rm d}u, f)\end{equation}
defines a function in $\mathcal{D}_{B,T}$. If $E$ is also bounded
\begin{equation}\label{z2}Z^2_E(f, t) : = Z^1_E(f, t) - \int\limits_E u\, \mu({\rm d}u)\end{equation}
defines an element in $\mathcal{D}_{B,T}$; we shall be primarily concerned with the $\mathcal{D}_{B,T}$-valued random variables of the form $Z^2_E(L, t)$.\bigskip

For every Borel set $E\subset B$ that is bounded below, the map $Z^1_E\colon \mathcal{D}_{B,T}\to \mathcal{D}_{B,T}$ is Borel. We are indebted to Mateusz Kwa\'snicki for sharing with us a direct proof of this fact (\cite{kwasnicki}); this argument replaces our previous, overly roundabout reasoning. Note that Borel measurability of $Z^2_E$ follows from Borel measurability of $Z^1_E$ as the former is a translation of the latter function. Let us then record these findings for the future reference. 
\begin{lemma}\label{Z12Borel}Fix $T>0$ and $E\subset H$ be a non-empty Borel set that is bounded below. Then, the transformation $Z^1_E\colon \mathcal{D}_{B,T}\to \mathcal{D}_{B,T}$ given by \eqref{z1} is Borel. \smallskip

When $E$ is also bounded, the same is true for $Z^2_E\colon \mathcal{D}_{B,T}\to \mathcal{D}_{B,T}$ given by \eqref{z2} being a translation of $Z^1_E$.\end{lemma}



\begin{remark}\label{convergence}As observed by Applebaum (\cite[Section 4]{App2}), for any $t\geqslant 0$ and for every sequence $(E_n)_{n=1}^\infty$ of Borel sets in the unit ball $B_1$ of $B$ such that $E^c_n = B_1 \setminus E_n$ ($n\in \mathbb N$) is bounded below and the sets $E_n$ decrease to $\{0\}$, the random variables 
$$Z^2_{E^c_n}(L, t) = \int\limits_{E^c_n} \hat{\pi}_t({\rm d}u, L),$$ converge almost surely as $n\to\infty$ to a random variable
$$Z^2_{B_1}(L,t):= \int\limits_{B_1} \hat{\pi}_t({\rm d}u, L)$$
(\cite[Section 2.3]{App1}; see also \cite[p.~80]{App2}). Moreover the above limit does not depend on the choice of $(E_n)_{n=1}^\infty$. \end{remark} 
Under this framework, one recovers the L\'evy--It\^o decomposition for $B$-valued L\'evy processes (see \cite[Theorem 4.1]{App2}, \cite[Theorem 2.1]{Det}, and \cite[Theorem 6.3]{riedlevangaans}).
\begin{theorem}[L\'evy--It\^o decomposition]\label{levyito}Let $B$ be a separable Banach space and let $(L_t)_{t\geqslant 0}$ be a $B$-valued L\'evy process with the corresponding Poisson random measure $$\{\pi_t(\cdot, L)\colon t\geqslant 0\}$$ that has intensity measure $\mu$. Then there are $b\in B$ and a Wiener process $W_Q$ with a~(possibly degenerate) covariance operator $Q$ such that
$$L_t = bt + W_Q(t) + \int\limits_{B_1} \hat{\pi}_t({\rm d}u, L) + \int\limits_{B\setminus B_1} {\pi}_t({\rm d}u, L)\quad (t\geqslant 0). $$\end{theorem}
We term $$Z_t = \int\limits_{B_1} \hat{\pi}_t({\rm d}u, L) + \int\limits_{B\setminus B_1} {\pi}_t({\rm d}u, L) = Z^1_{B\setminus B_1}(L, t)+ Z^2_{B_1}(L, t) \quad (t\geqslant 0)$$ the \emph{jump part} of $L$. The Wiener process $W_Q$ and the jump part $Z$ are independent (\cite[Theorem 6.3]{riedlevangaans}).

\section{Proof of Theorem A}
We consider a~$H$-valued L\'evy process $L=(L_t)_{t\geqslant 0}$ on a probability space $(\Omega, \mathcal {F}, \mathsf P)$ that is expressed in the L\'evy--It\^o decomposition as $L_t = bt + W_t + Z_t$ ($t\in [0,T]$), where $b\in H$, $W=(W_t)_{t\geqslant 0}$ is a (possibly degenerate) Wiener process and $(Z_t)_{t\geqslant 0}$ is the jump part of $L$. Let $X,\tilde{X}\colon \Omega\to \mathcal{D}_{H,T}$ be the Ornstein--Uhlenbeck processes solving \eqref{equations}. Moreover, we consider solutions $X^J$ and $\tilde{X}^J$ to the auxiliary equations without the Wiener part: 
\begin{equation}\label{equations2}\left\{\begin{array}{lcl}{\rm d}X^J_t &=& A X^J\, {\rm d}t + {\rm d}(L_t-W_t),\\ {\rm d}\tilde{X}^J_t &=& \tilde{A} \tilde{X}^J\, {\rm d}t + {\rm d}(L_t-W_t)\end{array}\right. \end{equation}
with the initial conditions $X^J(0)=\tilde{X}^J(0)=0$. Let us take a note that $X^J$ and $\tilde{X}^J$ (can be modified to) have c\`adl\`ag sample paths, which we will employ later.
\begin{proposition}The processes $X$ and $X^J + W$ have equivalent laws. \end{proposition}
\begin{proof}Let $(\mathcal{F}_t^W)_{t\geqslant 0}$ be the natural filtration of $W$ and let us consider the process
$$W^*_t = X_t - X^J_t = \int\limits_0^t A(X_s - X^J_s)\, {\rm d}s + W_t\quad \big(t\in [0,T]\big). $$
Then $W^*$ is the unique strong solution to ${\rm d}W^*_t = AW^*_t\, {\rm d}t + {\rm d}W_t$ with $W^*_0 = 0$ that is adapted to the filtration $(\mathcal{F}_t^W)_{t\geqslant 0}$. \smallskip 

Since $Z$ and $W$ are independent processes, so are $Z$ and $W^*$. By Girsanov's theorem (see \cite[Theorem 1]{Loges} for a version of Girsanov's theorem for $H$-valued processes; this is where we apply the hypothesis that the eigenvalues of the covariance operator are strictly positive as well as \eqref{technical} in the case where $H$ is innfinite-dimensional), there is a probability measure $\tilde{\mathsf P}$ for which $W^*$ is a Wiener process on $(\Omega, \mathcal F, \tilde{\mathsf P})$ with the same covariance operator as $W$ and so the laws $\mathsf P_W$ and $\mathsf P_{W^*}$ are equivalent. The processes $X^J$ and $W$ are independent. Let us observe that the processes $X^J$ and $W^*$ are independent too. Indeed, $W^*$ being adapted to $(\mathcal{F}_t^W)_{t\geqslant 0}$ is $\mathcal{F}^W$-measurable, and thus independent from $(X^J_t)_{t\geqslant 0}$ (see also \cite[Theorem II.6.3]{IW}).\smallskip

We are now in a position to apply Lemma~\ref{Lemma2new} to conclude that the laws $\mathsf{P}_{(X^J, W)}$ and $\mathsf{P}_{(X^J, W^*)}$ are equivalent. Consequently, the laws $\mathsf{P}_{X^J+W}$ and $\mathsf{P}_{X^J+W^*} = \mathsf P_X$ are equivalent as well, which completes the proof. \end{proof}

Thus, in order to establish Theorem A, it is enough to prove the following proposition.

\begin{proposition}\label{prop1}The processes $X^J +W$ and $\tilde{X}^J + W$ have equivalent laws. \end{proposition}
\begin{proof}We will demonstrate that the law of $X^J +W$ is absolutely continuous with respect to the law of $\tilde{X}^J +W$ as the other direction would be completely analogous. \smallskip 

For given $R > 0$ consider the set
$$\Omega_R:= \{\omega \in \Omega\colon \sup_{t\in [0,T]} \|AX_t(\omega) - \tilde{A}\tilde{X}_t(\omega)\| \leqslant R \} $$
and note that $\Omega_R\in \mathcal{F}$ (\emph{cf.}~Proposition~\ref{usefulprojections}). We may then consider the `truncated' process
$$W^R_t = W_t + \int\limits_0^t \big(AX_s - \tilde{A}\tilde{X}_s\big)\cdot \mathds{1}_{\Omega_R}\, {\rm d}s\quad (t\in [0,T]). $$
By hypothesis \eqref{technical}, we may apply Girsanov's theorem, so there is a probability measure $\mathsf P^R$ equivalent to $\mathsf P$, for which $W^R$ is a Wiener process on $(\Omega, \mathcal F,\mathsf P^R)$ with the same covariance operator as $W$; in particular the laws $\mathsf P_W$ and $\mathsf P_{W^R}$ are equivalent. Arguing as in the proof of Proposition~\ref{prop1}, we infer that the processes $X$ and $W^R$ are independent. By Lemma~\ref{Lemma2new} applied to $(\tilde{X}^J, W)$ and $(\tilde{X}^J, W^R$), we deduce that the processes $\tilde{X}^J+W$ and $\tilde{X}^J +W^R$ have equivalent laws.\smallskip 

Since
$$X^J_t - \tilde{X}^J_t = \int\limits_0^t \big( AX_s - \tilde{A}\tilde{X}_s\big) \,{\rm d}s  $$
we see that $\tilde{X}_t^J+W^R_t$ and $X^J_t + W_t$ agree on the set $\Omega_R$. It follows that for every set $E\in \mathsf{Bor}\, \mathcal{D}_{H,T}$ the condition $\mathsf P(\tilde{X}_t^J +W^R_t \in E)=0$ implies that for all numbers $R>0$ we have $\mathsf{P}\big( \Omega_R\cap (\tilde{X}^J_t + W)\big)=0$. The processes $X^J, \tilde{X}^J$ have c\`adl\`ag sample paths, which implies that they are bounded on bounded intervals. In particular, for any $\omega\in \Omega$ the function
$$t\mapsto  AX_t(\omega) - \tilde{A}\tilde{X}_t(\omega) \quad (t\in [0,T]),$$
is bounded, which means that $\Omega = \bigcup_{R>0}\Omega_R$. Thus $\mathsf{P}_{{X}^J + W}\ll \mathsf{P}_{\tilde{X}^J + W}$. \end{proof}

\section{Proof of Theorem B}
This time we consider a purely jump L\'evy process $L$, that is $L = Z$ using the notation of Theorem~\ref{levyito}. Let $X^J, \tilde{X}^J\colon \Omega\to \mathcal{D}_{H,T}$ be solutions to \eqref{equations}, which now take the form
\begin{equation}\left\{\begin{array}{lcl}\label{equations3}{\rm d}X^J_t& =& A X^J_t\, {\rm d}t + {\rm d}Z_t,\\ {\rm d}\tilde{X}^J_t &= &\tilde{A} \tilde{X^J}_t\, {\rm d}t + {\rm d}{Z}_t\end{array}\right. \quad (t\in [0,T])\end{equation}
with the initial conditions $X^J_0 = \tilde{X}^J_0 = 0$.

\begin{lemma}For any $t\in [0,T]$ and $E\in \mathsf{Bor}\, H\setminus \{0\}$ we have
\begin{equation}\label{XJZ}\mathsf P\big(\pi_t(E,X_J ) = \pi_t(E,Z) \big) = 1.\end{equation} \end{lemma}
\begin{proof} Since

$$X^J_t = A \int\limits_0^t X^J_s\, {\rm d} s + Z_t,$$
we have
$$\begin{array}{lcl}\{s\leqslant t\colon X^J_s - X^J_{s-}\in E \} &=& \{s\leqslant t\colon A \int\limits_0^s X^J_u\, {\rm d} u + Z_s -  A\int\limits_0^{s-} X^J_u\, {\rm d} u - Z_{s-}\in E\} \\
& = & \{s\leqslant t\colon  Z_s -   Z_{s-}\in E\} \end{array}$$
almost surely. \end{proof}
We may then derive the following conclusion.
\begin{corollary}\label{convergence2}
For every non-negative integer $k$, $t\in [0,T]$, and $E\in \mathsf{Bor}\, H\setminus \{0\}$ we have
$$\mathsf P_{X^J}\big(\{f\in \mathcal{D}_{H,T}\colon \pi_t(E,f) = k\} \big) = \mathsf P\big(\pi_t(E,X^J ) = k\big) = \mathsf P\big(\pi_t(E,Z ) = k\big).$$
In particular, for a sequence $(E_n)_{n=1}^\infty$ of Borel sets in the unit ball $H$, as in the statement of Remark~\ref{convergence}, the random variables $Z^2_{E^c_n}(X^J, t)$ converge almost surely to $Z^2_{H_1}(Z, t)$ as $n\to \infty$ $(t\in [0,T])$.\end{corollary}

For $t\in [0,T]$ and $f\in \mathcal{D}_{H,T}$ we set $$S(f,t) = \int\limits_0^t f(s)\,{\rm d}s.$$ Consequently, the assignment $f\mapsto S(f, \cdot)$ defines a function $\mathcal{D}_{H,T}\to \mathcal{D}_{H,T}$ as it takes continuous values. 
\begin{lemma}\label{ASBorel}For a bounded linear operator $V\colon H\to H$, the assignment $\mathcal{D}_{H,T}\to \mathcal{D}_{H,T}$ given by $f\mapsto VS(f, \cdot)$ $(f\in \mathcal{D}_{H,T})$ is Borel. \end{lemma}
\begin{proof}Observe that if a sequence $(f_n)_{n=1}^\infty$ in $\mathcal{D}_{H,T}$ converges to some $f\in \mathcal{D}_{H,T}$, then for almost all $s\in [0,T]$ (with respect to the Lebesgue measure) we have $f_n(s)\to f(s)$ as $n\to\infty$.
Since the sequence $(f_n)_{n=1}^\infty$ is bounded with respect to the supremum norm, by the dominated convergence theorem (for Bochner-integrable functions) we conclude that for all $t\in [0,T]$
$$\int\limits_0^t f_n(s)\,{\rm d}s \to \int\limits_0^t f(s)\,{\rm d}s $$
as $n\to\infty$. Consequently, the map $\Phi(f)= VS(\cdot, f)$ ($f\in \mathcal{D}_{H,T}$) is Borel-measurable because it is continuous as a map from $\mathcal{D}_{H,T}$ with the Skorohod topology to $\mathcal{D}_{H,T}$ (actually even to $C([0,T], H)$) endowed with the topology of pointwise convergence. Indeed, by Proposition~\ref{usefulprojections}, $\Phi$ is Borel because for each $t\in [0,T]$, the map $\Phi(\cdot)(t)\colon \mathcal{D}_{H,T}\to H$ is continuous, hence Borel. \end{proof}

\begin{definition}For a bounded linear operator $V\colon H\to H$ we define
$$\Xi_V = \big\{f\in \mathcal{D}_{H,T}\colon \lim_{n\to\infty} \|f(t) - VS(f,t) -Z^1_{H\setminus H_1}(f,t) - Z^2_{E^c_n}(f, t)\|=0\text{ for all }t\in [0,T] \big\}.$$
\end{definition}

By Lemma~\ref{ASBorel}, the assignment $f\mapsto VS(f, \cdot)$ ($f\in \mathcal{D}_{H,T}$) is Borel measurable. Similarly, by Lemma~\ref{Z12Borel}, the assignments $f\mapsto Z^1_E(f, \cdot), Z^2_E(f, \cdot)$ ($f\in \mathcal{D}_{H,T}$) are Borel measurable for every bounded below set $E\subset H$ (in the latter case, $E$ is assumed to be additionally bounded). Let us invoke Proposition~\ref{usefulprojections} to see that in order to establish measurability of $\Xi_V$ it is enough to show measurability of $p_t[\Xi_V]$ for each $t\in [0,T]$. We have thus proved the following proposition.
\begin{proposition}For a bounded linear operator $V\colon H\to H$, the set $\Xi_V$ is Borel with respect to the Skorohod topology on $\mathcal{D}_{H,T}$.\end{proposition}

\begin{lemma}\label{lemma43}$\mathsf P_{X^J}(\Xi_A) = 1 = \mathsf{P}_{\tilde{X}^J}(\Xi_{\tilde{A}}).$ \end{lemma}
\begin{proof}By Corollary~\ref{convergence2}, $Z^1_{H\setminus H_1}(X^J,t) + Z^2_{E^c_n}(X^J, t)$ converges almost surely as $n\to\infty$ to $Z(L,t)$ ($t\in [0,T]$). Thus, $\mathsf P_{X^J}(\Xi_A)$ is equal to
$$\begin{array}{lcl}
\quad \mathsf P\big(X^J_t - AS(X^J, t) = Z(X^J, t)\; (t\in [0,T]) \big) = \mathsf P\big(X^J_t - A\int\limits_0^t X^J_s{\rm d}s = Z_t\; (t\in [0,T]) \big)= 1.\end{array} $$
The same proof applies for $\Xi_{\tilde{A}}$.\end{proof}

We are now ready to prove Theorem B.
\begin{proof}[Proof of Theorem B]Assume that $\mathsf P_{X^J} \ll \mathsf P_{\tilde{X}^J}$. By Lemma~\ref{lemma43}, $\mathsf P_{X^J}(\Xi_A) = 1 = \mathsf{P}_{\tilde{X}^J}(\Xi_{\tilde{A}})$. Consequently, by absolute continuity we must have $\mathsf P_{X^J}(\Xi_{\tilde{A}})=1$, which means that
$$X^J_t = \tilde{A} \int\limits_0^t X^J_s\,{\rm d}s + Z_t \quad \big(t\in [0,T]\big)$$
almost surely. We have thus proved that $X^J$ solves the stochastic differential equation ${\rm d}Y_t = AY_t {\rm d}t + {\rm d}Z_t$, so by the uniqueness of solutions, $X^J = \tilde{X}^J$ almost surely.\end{proof}

\subsection{Closing remarks}
In the case of stochastic processes in infinite dimensions it is customary to work with generators of infinitesimal semigroups rather than merely bounded linear operators. However, the proof methods employed in this paper required the operators $A$ and $\tilde{A}$ appearing in \eqref{equations} to be bounded (\emph{cf}.~the proofs of Proposition~\ref{prop1} and Lemma~\ref{ASBorel}). It is thus natural to ask whether Theorems A and B have their counterparts in the setting of generators of infinitesimal semigroups too. From this point of view, it is also desirable to investigate classes of those Feller processes for which analogous results can be established.

\subsection*{Acknowledgements} We wish to express our gratitude to Mateusz Kwa\'snicki (Wroc{\l}aw) for sharing with us a direct proof of Lemma~\ref{Z12Borel}. Furthermore, we are greatly indebted to the anonymous referee for spotting the need for the technical assumption \eqref{technical} that is indeed required when the underlying Hilbert space is infinite-dimensional.

\end{document}